\documentclass{article}
\usepackage{amsmath,amssymb,amsthm,amsmath}
\usepackage{graphicx, caption, booktabs}
\usepackage{fullpage}
\usepackage[utf8]{inputenc}

\def\ds{\displaystyle}
\def\G{\mathcal{G}}
\def\A{\mathcal{A}}

\newcommand{\ex}{\mathbf{E}}

\newtheorem{theorem}{Theorem}
\newtheorem{lemma}[theorem]{Lemma}
\newtheorem{corollary}[theorem]{Corollary}

\def\ds{\displaystyle}

\def\G{\mathcal{G}}
\def\T{\mathcal{T}}

\def\N{\mathcal{N}}

\title{On the expected number of perfect matchings in \\  cubic planar graphs}

\author{
	Marc Noy
	\thanks{
		Universitat Polit\`ecnica de Catalunya, Department of Mathematics and Institute of Mathematics
		E-mail: {\tt marc.noy@upc.edu}. 
		Supported by the Ministerio de Econom\'{i}a y Competitividad grant MTM2017-82166-P.
	}
	\and
	Cl\'ement Requil\'e
		\thanks{
		Technische Universit\"at Wien, Institute of Discrete Mathematics and Geometry.
		E-mail: {\tt clement.requile@tuwien.ac.at}. 
		Supported by the Special Research Program F50 \textit{Algorithmic and Enumerative Combinatorics} of the Austrian Science Fund.
	}
	\and
	Juanjo Ru\'e
	\thanks{
		Universitat Polit\`ecnica de Catalunya, Department of Mathematics and Institute of Mathematics. Centre de Recerca Matemàtica and Barcelona Graduate School of Mathematics.
		E-mail: {\tt juan.jose.rue@upc.edu}. 
		Supported by the Ministerio de Econom\'{i}a y Competitividad grant MTM2017-82166-P.
	}
}

\begin{document}

\maketitle

\begin{abstract}
A well-known conjecture by Lovász and Plummer from the 1970s asserted that a bridgeless cubic graph has exponentially many perfect matchings. 
It was solved in the affirmative by Esperet et al. (Adv. Math. 2011). %who showed that the number of perfect matchings is at least  $2^{n/3656} \approx 1.0002^n$. 
On the other hand, Chudnovsky and Seymour  (Combinatorica 2012) proved the conjecture in the special case of  cubic \emph{planar} graphs. 
In our work we consider random bridgeless cubic planar graphs with the uniform distribution on graphs with $n$ vertices. 
Under this model we show that the expected number of perfect matchings in \emph{labeled} bridgeless cubic planar graphs is asymptotically  $c\gamma^n$, where $c>0$ and  $\gamma \sim 1.14196$ is an explicit algebraic number. 
We also compute the expected number of perfect matchings in (non necessarily bridgeless) cubic planar graphs and provide lower bounds for \emph{unlabeled} graphs.
Our starting point is a correspondence between counting perfect matchings in rooted cubic planar maps and the partition function of the Ising model in rooted triangulations. 
\end{abstract}

%----------------------------------------------------
\section{Introduction}
%----------------------------------------------------

A perfect matching in a graph  with $2n$ vertices is a set of $n$  edges with no common endpoints, thus covering all the vertices. 
A bridge (also called an isthmus) is an edge whose removal increases the number of connected components. 
A bridgeless graph is a connected graph  without bridges, and it is cubic if every vertex has degree 3. A graph is bipartite if its vertex set can be divided into 
 two subsets $A$ and $B$ such that every edge joins a vertex in $A$ to one in $B$. 
Petersen proved in 1891 that a bridgeless cubic graph contains at least one perfect matching \cite{petersen}. 
In the 1970s Lovász and Plummer \cite{lovasz} showed that a cubic bipartite graph with $2n$ vertices has at least $(4/3)^n$ perfect matchings. 
% The proof is a simple application of  van der Waerden's conjecture (proved in 1981 independently by Falikman and Egorycev)  on the permanent of a doubly stochastic matrix. 
Then they conjectured that a  bridgeless cubic graph has  exponentially many perfect matchings, that is, contains at least $C^n$ perfect matchings for some constant $C>1$. 
The problem of counting perfect matchings has been much studied in combinatorics and has connections to problems in molecular chemistry (the stability of a molecule is related to the number of perfect matchings of its associated graph), statistical physics (counting perfect matchings plays a key role in the solution of the 2-dimensional Ising model of ferromagnetism), and computer science (as it is related to the complexity of computing the permanent of a matrix); see for instance \cite[Section 8.7]{lovasz} and \cite{valiant}.

For many years the best lower bound on the number of perfect matchings in bridgeless cubic graphs was only linear in $n$.  
A barely superlinear bound was proved in \cite{esperet2}, and finally the conjecture was fully solved in 2011 by Esperet, Kardoš, King, Král and Norine \cite{esperet}.
The proof uses in particular the concept of the matching polytope of a graph, which is the convex hull of the indicator vectors of its perfect matchings in the edge space.
The lower bound from \cite{esperet} is $2^{n/3656} \approx 1.0002^n$.
It is natural to expect that a typical bridgeless cubic graph has many more perfect matchings than those guaranteed by the former lower bound.
This is indeed the case: as shown by Bollobás and McKay \cite{bollobas}, the number of perfect matchings in random cubic graphs is concentrated around the expected value $\sqrt{2}e^{1/4}(4/3)^n$, where the probability model is the uniform distribution on labeled cubic graphs with $2n$ vertices. 
Here it is not assumed that graphs are bridgeless but it is known that a random cubic graph is bridgeless asymptotically almost surely (a.a.s.), that is, with probability tending to 1 as $n\to\infty$.

In this paper we are interested in cubic \emph{planar} graphs. 
We recall that a graph is planar if it admits an embedding in the plane without edge-crossings. 
Tait (known for his work, among other topics, on thermodynamics and knot theory) showed in 1880 that the statement of the Four color Theorem (4CT) is equivalent to the fact that the edges of a bridgeless cubic planar graph can be decomposed into three disjoint perfect matchings. 
Tait also conjectured \cite{tait} that a bridgeless cubic planar graph contains a Hamiltonian cycle, a conjecture that if true  would have implied the 4CT. 
Tait's conjecture turned out to be false, the first counterexample being found by Tutte \cite{tutteHam}.

The Lovász-Plummer conjecture was solved  for the special case of bridgeless cubic \emph{planar} graphs by Chudnovsky and Seymour \cite{chudnovsky}, who showed that a bridgeless cubic planar graph has at least $2^{n/327989376}$ perfect matchings; their proof uses the 4CT in an essential way. 
Motivated by this result, we study the number of perfect matchings in random cubic planar graphs. 
The theory of random planar graphs was initiated in \cite{DVW}, but it was not until the work of Giménez and Noy \cite{gn} that precise results could be obtained. 
Since then many parameters of random planar graphs and related families of graphs have been studied in depth (see \cite{noy} for a survey of this active area).  
Cubic planar graphs were first enumerated in \cite{bklm2007}; they were further analyzed in \cite{cubic-revisited} and we draw extensively on the combinatorial and analytic techniques developed there. 

Our main result gives a precise estimate on the expected number of perfect matchings for labeled bridgeless cubic planar graphs. 
We also give a lower bound for unlabeled bridgeless cubic planar graphs, where an unlabeled graph is an isomorphism class of labeled graphs.  
The model we consider is the uniform distribution on each respective class of cubic graphs with $2n$ vertices. 
Results under this model are in general more difficult to obtain than for arbitrary cubic graphs, for which the so-called configuration model (see \cite{bollobasRG}) provides a convenient framework for applying probabilistic methods.

\begin{theorem}\label{th:bridgeless}
Let $X_n$ be the  number of perfect matchings in a random (with the uniform distribution) labelled bridgeless cubic planar graph with $2n$ vertices.  
Then 
$$
	\ex (X_n) \sim b\gamma^n,
$$
where $b>0$ is a constant and $\gamma \approx 1.14196$ is an explicit algebraic number.

If $X^u_n$ is the same random variable defined on unlabeled bridgeless cubic planar graphs, then 
$$
	\ex (X^u_n) \ge 1.119^n.
$$ 
\end{theorem}

We obtain a similar result for general, non necessarily bridgeless, random cubic planar graphs. 

\begin{theorem}\label{th:all}	
Let $Y_n$ be the  number of perfect matchings in a random (with the uniform distribution) labelled  cubic planar graph with $2n$ vertices. 
Then 
$$
	\ex (Y_n) \sim c \delta^n,
$$
where $c>0$ is a constant and $\delta \approx 1.14157$ is an explicit algebraic number.
	
If $Y^u_n$ is the same random variable defined  on unlabeled  cubic planar graphs, then 
$$
	\ex (Y^u_n) \ge 1.109^n.
$$ 
\end{theorem}

We remark that the constants $\gamma$ and $\delta$ in the previous statements are smaller than the constant $4/3$ from \cite{bollobas}, hence random cubic planar graphs appear to have exponentially fewer perfect matchings than cubic graphs (we cannot claim this rigorously since, as discussed in the concluding remarks, we are not able to estimate the variance of $X_n$ and show concentration around the expected value).

The enumeration of \emph{unlabeled} planar graphs is still an open problem, as well as that of unlabeled cubic planar graphs.
This is the reason why in the unlabelled  case we only give lower bounds on the expected number of perfect matchings.
Also, it follows from the results in \cite{cubic-revisited} that a random labeled cubic planar graph has no perfect matching a.a.s. 
This is because a.a.s.\ it contains $K_{1,3}$ as an induced subgraph, which is clearly an obstruction for the existence of a perfect matching; in fact it contains linearly many copies of it \cite{cubic-revisited}.

We prove analogous results for rooted planar maps.
A (planar) map is a connected planar multigraph with a fixed embedding in the plane, and it is rooted if an edge (the root edge) is distinguished and given a direction. 
All maps considered in this paper are rooted.
The theory of map enumeration was started by Tutte in a series of landmark papers \cite{tutte1,tutte2,tutte3}, motivated by the then unsolved Four color Problem (see Tutte's account in his fascinating mathematical autobiography \cite{tutteBook}). 
Since then, the theory has been widely developed and extended to maps on arbitrary surfaces.
Relevant connections have been found between map enumeration and other areas, including Riemann surfaces, factorizations of permutations, Brownian motion, random matrices and quantum gravity; see the classical paper \cite{brezin} and the monographs \cite{lando,eynard}. 
There has been an increasing interest in counting maps equipped with a distinguished global structure such as a spanning tree, a $q$-coloring or an orientation, partly because of its connections with probability theory and statistical physics.
One of the earlier results in this area is by Tutte \cite{tutte2}, who showed that the expected number of bridgeless cubic maps with a distinguished Hamiltonian cycle grows like $n^{-3} 16^n$, up to a multiplicative constant.  
Later Mullin \cite{mullin} proved that the number of arbitrary maps with a distinguished spanning tree also grows like $n^{-3} 16^n$; 
surprisingly, these two counting problems have essentially the same solution as the product of two consecutive Catalan numbers.
On the other hand, the number of 2-connected quadrangulations (every face has degree 4) equipped with a so-called 2-orientation grows like $n^{-4}8^n$, and the number of 3-connected triangulations (every face has degree 3) equipped with a 3-orientation grows like $n^{-5}16^n$ (see for instance \cite{baxter}).
Thus, as opposed to natural classes of maps, where the subexponential term is of the form $n^{-5/2}$ (see \cite{universal} for an analytic perspective on this universality phenomenon), maps with a distinguished global structure present a variety of subexponential terms. 

For cubic maps equipped with a perfect matching, our estimates are of the form $n^{-5/2}\alpha^n$ for some constant $\alpha>0$ that depends on the particular class of maps; for instance $\alpha=24$ for all cubic maps, and $\alpha = 512/27$ for bridgeless cubic maps. 
A consequence of this estimate is that the expected number of perfect matchings grows like a pure exponential $\gamma^n$, an asymptotic behavior that to our knowledge has not been observed before in this context. 
An explanation comes from the fact that the associated generating functions are algebraic, whereas those counting global structures on maps mentioned before are D-finite (a function is D-finite, or holonomic, if it is the solution of a linear differential equation with polynomial coefficients)
but not algebraic;
there are examples which are not even D-finite, like 4-regular maps equipped with an Eulerian orientation, whose growth $\mu^n/(n\log n)^2$ prevents the associated generating function from being D-finite \cite{eulerian}. 

Our proofs are based on the methods of analytic combinatorics \cite{fs} applied to cubic planar graphs developed in \cite{cubic-revisited}, together with several combinatorial bijections.
In the preliminary Section \ref{sec:prelim} we recall the basic tools needed from analytic combinatorics and generating functions of planar maps. 
In Section \ref{sec:Ising} we establish the connection between perfect matchings in cubic maps and the partition function of the Ising model
in triangulations, which is the starting point of our research.
In Section \ref{sec:graphs} we transfer the results from cubic maps to cubic planar graphs using combinatorial decompositions and polynomial equations satisfied by the associated algebraic generating functions.
In Section \ref{sec:results} we prove our main results, and in Section \ref{sec:bijections} we provide bijective proofs for the unexpectedly simple formulas we have found for the number of cubic maps and of bridgeless cubic maps equipped with a perfect matching. 
In the concluding section we discuss an interesting connection with the recent solution of the problem of counting 4-regular planar graphs \cite{4-regular,4-regular-asympt}.
We also argue why in this context it appears to be difficult to say more on the distribution of the number of perfect matchings beyond its expected value.

%----------------------------------------------------
\section{Preliminaries}\label{sec:prelim}
%----------------------------------------------------

For basic graph theory concepts and terminology we refer to \cite{diestel}.
Unless specified otherwise, graphs are simple and labeled, that is the vertex set is $V=\{1,\dots,n\}$ and there are no loops or multiple edges. 
A graph is connected if every two vertices are joined by a path, and is $k$-connected if it has at least $k+1$ vertices and cannot be disconnected by removing less than $k$ vertices.  

A \emph{map} is a connected planar multigraph with a specific embedding in the plane. The faces of a map are the connected components of the complement of the set of vertices and edges. An edge is incident with two faces, which are the same for a bridge. 
An embedding of a connected graph with $n$ edges in an orientable surface can be specified combinatorially by giving for each vertex the cyclic ordering of the half-edges around it. 
In algebraic terms, an embedding is given by a pair of permutations $(\alpha,\sigma)$ on the set of $2n$ half-edges, where $\alpha$ is an involution without fixed points and $\{\alpha,\sigma\}$ act transitively on the set of half-edges \cite{lando}. 
The faces correspond to the cycles of $\phi=\alpha\sigma$, and the genus of the surface is determined by the number of cycles in $\sigma$ and $\phi$. %If  in addition edges are signed, then one can encode embeddings in non-orientable surfaces as well. 

All maps considered in this paper are rooted, that is, an edge is marked and given a direction. We call the face on the right side of the root edge the \emph{outer face}.
A map is \emph{simple} if it has no loops nor multiple edges. 
Rooted maps have no automorphism, hence all vertices, edges and faces of the embedding are distinguishable. 
A map is 2-connected if it has no loops and no cut vertices, and 3-connected if it is 3-connected as a graph and has no multiple edges (this is connectivity in the sense of Tutte).
This is however not the case for regular maps, which can have multiple edges. 
The dual $M^*$ of a map $M$ has the faces of $M$ as vertices, and every edge $e$ of $M$ gives rise to an edge $e^*$ of $M^*$ between the two faces incident with $e$. 
Note that a bridge in $M$ corresponds to a loop in $M^*$ and conversely.

A map is cubic if it is 3-regular, and it is a triangulation if every face has degree 3. 
By duality, cubic maps are in bijection with triangulations.
And since duality preserves graph connectivity, $k$-connected cubic maps are in bijection with $k$-connected triangulations, for $k=2, 3$. 
Note that a general triangulation can have loops and multiple edges, and that a simple triangulation, not reduced to the single triangle, is necessarily 3-connected. 
The size of a cubic map is defined as the number of faces minus 2, a convention that simplifies the forthcoming algebraic computations. 
Then a cubic map of size $n$ has $2n$ vertices and $3n$ edges. 

We use  generating functions, both ordinary (for maps) and exponential (for labeled graphs). 
For a class of maps $\A$, we let $A(z) = \sum A_n z^n$ be the associated generating function, where $A_n$ is the number of maps in $\A$ with for example $n$ edges; we say in this case that the variable $z$ `marks' the number of edges. 
For a class $\G$ of labeled graphs, the associated exponential generating function is $G(x) = \sum G_n \frac{x^n}{n!}$, where now $G_n$ is the number of graphs in $\G$ on $n$ vertices, and in this case we say that $x$ marks the number of vertices. 
The $n$-th coefficient of a power series $f(z)$ is denoted by $[z^n]f(z)$.

We will need the generating function of 3-connected cubic maps, which is related to the generating function $T(z)$ of simple triangulations. 
The latter was obtained by Tutte \cite{tutte3} and is an algebraic function given by 
\begin{equation}\label{eq:Tu}
	T(z) = U(z)\left(1- 2U(z)\right),
\end{equation}
where $z = U(z)(1-U(z))^3$, and $z$ marks the number of vertices minus two.
As shown in \cite{tutte2}, the unique singularity of $T$, coming from a branch point, is located at $\tau = 27/256$ and $T(\tau) = 1/8$. 
The singular expansion of $T(z)$ near $\tau$ is
\begin{equation*}\label{eq:sing_T3}
	T(z) =\frac{1}{8} - \frac{3}{16}Z^2 + \frac{\sqrt{6}}{24}Z^3+O(Z^4),
\end{equation*}
where $Z = \sqrt{1-z/\tau}$. 
Notice that $\tau$ is a finite singularity, in the sense that $T(\tau) =1/8 < \infty$. 

The generating function $M_3(z)$ of 3-connected cubic maps, where $z$ marks the number of faces minus 2 is equal to 
\begin{equation}\label{eq:3-connected-edges}
	M_3(z) = T(z) - z.
\end{equation}
This follows directly from the duality between cubic maps and triangulations, which exchanges vertices and faces. 
The subtracted term corresponds to the single triangle, which is not considered to be 3-connected.

Given a map $N$ with root edge $st$, and a directed edge $e=uv$ of another map $M$, the \emph{replacement} of $e$ with $N$ is the map obtained from $M$ by performing the following operation.
Subdivide the edge $uv$ twice producing a path $uu'v'v$, remove the edge $u'v'$, and identify $u'$ and $v'$, respectively, with vertices $s$ and $t$ of $N-st$.
Notice that if $M$ and $N$ are cubic and planar, so is the resulting map.
Adapting directly the proof from \cite{cubic-revisited} for cubic planar graphs, one finds that cubic maps are partitioned into five subclasses, as defined below, and where $st$ denotes the root edge of a cubic map $M$.
\begin{itemize}
	\item $\mathcal{L}$ (Loop). The root edge is a loop.
	\item $\mathcal{I}$ (Isthmus). The root edge is an isthmus (an alternative name for a bridge).
	\item $\mathcal{S}$ (Series). $M-st$ is connected but not 2-connected.
	\item $\mathcal{P}$ (Parallel). $M-st$ is 2-connected but $M - \{s,t\}$ is not connected.
	\item $\mathcal{H}$ (polyHedral). $M$ is obtained from a 3-connected cubic map by possibly replacing each non-root edge with a cubic map whose root edge is not an isthmus.
\end{itemize}	
\noindent
We also define the subclass $\mathcal{D} = \mathcal{L} \cup \mathcal{S} \cup \mathcal{P} \cup \mathcal{H}$ of cubic maps whose root edge is not an isthmus. 

We use complex analytic tools to obtain estimates for the coefficients of generating functions. 
In our case, all the functions involved are algebraic and analytic in a neighborhood of the origin. 
Given an algebraic generating function $A(z)$, its minimal polynomial is the polynomial $P(y,z)$ of minimum degree in $y$ such that $P(A(z),z) = 0$. 
The discriminant $\Delta(z)$ is the resultant of $P(y,z)$ and its derivative with respect to $y$. 
The singularities of $A(z)$ (the points in the complex plane where $A(z)$ ceases to be analytic) are necessarily among the roots of $\Delta(z) = 0$. 
A dominant singularity is a singularity of minimum modulus, and a singular expansion of $A(z)$ is an expansion locally around its dominant singularity.
If the coefficients of $A(z)$ are non-negative, as is the case in counting generating functions, by Pringheim's theorem the radius of convergence $\rho$ is a dominant singularity. 
In fact, in the map counting functions involved in our analysis, $\rho$ is the unique dominant singularity (the functions counting graphs are even, so that $-\rho$ is a dominant singularity too). 
Together with the fact that the singularities come from a critical point, this implies that $A(z)$ is analytic in a dented domain (doubly dented when $\pm\rho$ are both singular points) of the form 
$$
	{\bf D}(\rho,\epsilon,\phi) = \{z \colon |z| < |\rho|+\epsilon, \arg(z/\rho-1) > \phi \}.
$$
Then one can estimate precisely the coefficients $[z^n]A(z)$ by means of Cauchy's integral formula 
$$
	[z^n]A(z) = \frac{1}{2\pi i} \int_C \frac{A(z)}{z^{n+1}}dz,
$$
integrating along a suitable contour $C$, and obtain the so-called `transfer theorem' \cite[Corollary VI.1]{fs}. 
We remark that since our functions are algebraic, in the statement below the singular expansion is in fact a Puiseux expansion. We state below 
the transfer theorem for singularities of exponent $3/2$, which are those appearing in our work.

\begin{lemma}\label{le:transfer}
	
Assume $A(z)$ is analytic in a dented domain around its unique dominant singularity $\rho >0$, and has a local  expansion at $\rho$ of the form 
$$
	A(z) = a_ 0 + a_2Z^2 + a_3 Z^3 + O(Z^4),
$$
where $Z = \sqrt{1-z/\rho}$ and $a_3>0$.
Then 
$$
	[z^n] A(z) \sim \frac{3a_3}{2 \sqrt{\pi}}  n^{-5/2} \rho^{-n},
$$
where the notation $a_n \sim b_n$ means $\lim_{n\to\infty} a_n/b_n = 1$. 
\end{lemma}

A straightforward generalization of the transfer theorem covers the situation when there are multiple singularities in the circle of convergence. 

In order to determine the dominant singularities of an algebraic function, we proceed as follows. 
Suppose $P(y,z)$ is the minimal polynomial of $A(z)$.   
The discriminant is the univariate polynomial given by the resultant 
$$
	 \mathrm{Res}\left(P(y,z),\frac{\partial P(y,z)}{\partial y}, y\right).
$$
The singularities of $A(z)$ must necessarily be among the roots of the discriminant, and the dominant singularity is as a rule the smallest positive root of one of the factors of the discriminant. 
However, sometimes a root $z_0$ must be excluded because either it corresponds to a branch of the algebraic function $A(z)$ that does not have non-negative coefficients, or because the growth rate $z_0^{-n}$ is either too large or too small for combinatorial reasons.
% In the proofs of our main results such an ambiguity does not occur. 

The reader can find more details in \cite[Section 2]{cubic-revisited}. 
In particular, \cite[Lemma 15]{cubic-revisited} applies directly to our analysis, and the function $T(z)$ in this lemma is exactly the same we have defined earlier. 
We will deal with algebraic functions $A(z)$, whose expansion at the dominant singularity is as above.
To check that the conditions of \cite[Lemma 15]{cubic-revisited} hold is a routine task and we will omit most of the details.

We conclude this section with a fixed-point lemma  which is needed to guarantee that certain system of equations have unique combinatorial solutions; see the discussion in \cite[Section 2.2.5]{DrmotaBook}.

\begin{lemma}\label{positive}
	Let $y_1(z), \dots, y_m(z)$ be power series satisfying the system of equations
	$$
	\begin{array}{lcl}
		y_1 &=& F_1(z, y_1, \dots , y_m), \\
		y_2 &=& F_2(z, y_1, \dots , y_m),\\
		&\vdots& \\
		y_m &=& F_m(z, y_1, \dots , y_m),\\
	\end{array}
	$$
	where the $F_i $ are power series in the variables indicated.
	
	Assume that for each $i$, $F_i$ has non-negative coefficients and is divisible by $z$.
	Assume further that there exists a solution $\mathbf{\widehat{y}}=(y_1(z),\dots,y_m(z))$ to the system which is not identically $0$ for all $i$.
	Then this is the unique  solution of the system with non-negative coefficients.
\end{lemma}

%---------------------------------------------------
\section{The Ising partition function on triangulations and perfect matchings in cubic maps}\label{sec:Ising}
%----------------------------------------------------

The Ising partition function of a graph $G$ is  defined as follows. 
Given a 2-coloring $c \colon V(G)\to \{1,2\}$ of the vertices of $G$, not necessarily proper, let $m(c)$ be the number of monochromatic edges in the coloring. 
Then define 
$$
	p_G(u) = \sum_{c \colon V(T)\to \{1,2\}} u^{m(c)}.
$$
The same definition applies for rooted maps, using the fact that in a rooted map the vertices are distinguishable. 
The physical intuition behind the Ising model for ferromagnetism is that the  colors $\pm1$ represent the possible spin of a site (vertex) in a system, and a coloring is an assignment of a spin value to each site.
Then the partition function is directly related to the energy of the system.
	
Suppose $T$ is a triangulation with $2n$ faces. 
Since in a 2-coloring every face of $T$ has at least one monochromatic edge, the number of monochromatic edges is at least $n$. 
The lower bound can be achieved taking the dual edge-set of a perfect matching in a cubic map.
We show next that perfect matchings of a cubic map $M$ with $2n$ vertices are in bijection with 2-colorings of the dual triangulation $M^*$ with exactly $n$ monochromatic edges, in which the color of the root vertex is fixed (see Figure \ref{fig:Ising}). 
Note that loops are allowed in general triangulations and a ``degenerate'' face can consist of a loop and a bridge sharing one vertex. 

\begin{figure}[htb]
	\centering
	\begin{tabular}{cccccc}
	\raisebox{.3cm}{\includegraphics[scale=1]{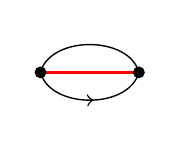}} &
	\raisebox{.3cm}{\includegraphics[scale=1]{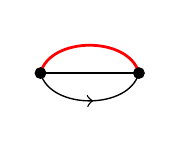}} &
	\raisebox{.3cm}{\includegraphics[scale=1]{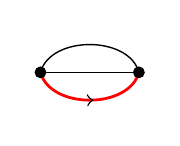}} &
	\raisebox{.4cm}{\includegraphics[scale=1]{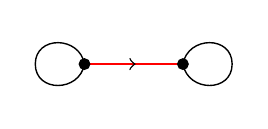}} &
	\raisebox{.4cm}{\includegraphics[scale=1]{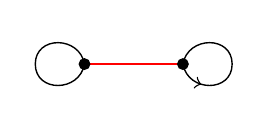}} &
	\includegraphics[scale=1]{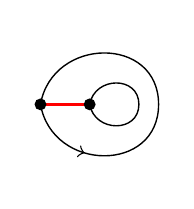} \\
	\includegraphics[scale=1]{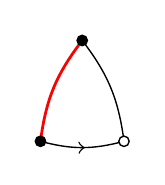} &
	\includegraphics[scale=1]{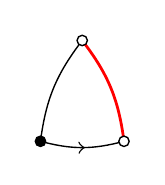} &
	\includegraphics[scale=1]{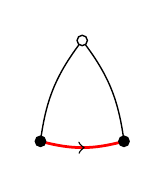} &
	\raisebox{.1cm}{\includegraphics[scale=1]{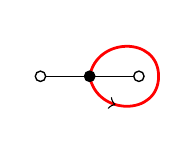}} &
	\raisebox{.1cm}{\includegraphics[scale=1]{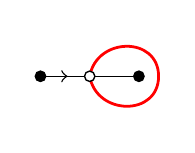}} &
	\raisebox{.1cm}{\includegraphics[scale=1]{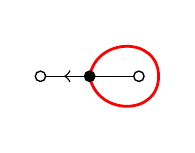}}
	\end{tabular}
	\caption{
		We illustrate the duality between the six cubic maps on two vertices with a distinguished perfect matching (top), and their counterparts (bottom): the six, up to changing the color of the root vertex, bicolored triangulations on two faces with one monochromatic edge.
		Both the edges of the matchings and the monochromatic edges are shown  in red.
	}
	\label{fig:Ising}
\end{figure}	  

\begin{lemma}\label{le:Ising-PM}
	Let $M$ be a rooted cubic planar map and $T=M^*$ its dual triangulation. 
	There is a bijection between perfect matchings of $M$ and 2-colorings of $T$, with exactly $n$ monochromatic edges, in which the color of the root vertex of $T$ is fixed. 
\end{lemma}

\begin{proof}
Let $M$ have $2n$ vertices, so that $T$ has $2n$ triangular faces. 
Let $c \colon V(T) \to \{1,2\}$ be a 2-coloring with $n$ monochromatic edges. 
Define the edge-set $A = \{e \in E(M) \colon  e^* \hbox{ is monochromatic in $T$}\}$. 
Since every face of $T$ is incident with exactly one monochromatic edge, every vertex of $M$ is incident with exactly one edge of $A$. 
Hence $A$ defines a perfect matching of $M$.

Now we construct the inverse mapping. 
Let $A$ be the edge-set of a perfect matching in $M$, and let $B = A^* =  \{e^* \colon e \in A\}$ be the dual edge-set. 
Let $T[B]$ be the graph induced on $T$ by $B$, and let $C_1,\dots, C_s$ be the connected components of $T[B]$. 
We define a graph $G$ having as vertex-set $\{C_1,\dots,C_s\}$, and an edge $C_iC_j$ whenever a vertex of $C_i$ is adjacent to a vertex of $C_j$ (necessarily through an edge not in $B$). 
We next prove that $G$ is a bipartite graph. 
If there are multiple edges we can ignore them. 
However, we have to show that there are no loops, which are cycles of length one. 
This is taken care of in what follows.
 
Assume to a contradiction that $G$ contains a cycle $L$ of odd length $t\ge1$. 
The edges of $L$ in circular order must be of the form $e_1F_1e_2F_2\dots e_tF_t$, where the $e_i$ are not in $B$ and the $F_i$ are induced paths in $B$.
The subgraph $S$ induced by $L$ and its interior is a near-triangulation, that is, every face except the outer face is of degree 3. 
Let $k$ be the total number of vertices in $L$, which is the degree of the outer face of $S$, and let $n$ be the number of vertices in the interior of $L$. 
An elementary counting argument shows that the number of internal faces of $S$ is equal to $2n + k - 2$.
Exactly $k-t$ of these faces contain an edge of $B$ belonging to $L$. 
The number of remaining faces is then $2n + k - 2 - (k - t) = 2n - 2 + t$, which is an odd number since $t$ is odd. 
It follows that it is not possible for all inner triangular faces to have exactly one edge in $B$.
Observe that in particular we have shown that $G$ has no loops, that is, two vertices in the same component joined by an edge not in $B$.

Since $G$ is bipartite we can properly color the components $C_i$ with colors 1 and 2.
The 2-coloring of $T$ is defined by assigning color $1$ to all vertices in a component colored 1, and the same for color 2,  with the additional property that the component containing the root vertex is colored 1.
Then the monochromatic edges are precisely the edges in $B$, as was to be proved.
\end{proof}

The generation function of the Ising partition of triangulations is defined as the generating function 
$$
	P(z,u) = \sum_{T\in \T} p_T(u)z^n,
$$ 
where $\T$ is the class of rooted triangulations and the variable $z$ marks the number of vertices minus 2, in accordance with the convention for cubic maps.

An expression for $P$ was obtained by Bernardi and Bousquet-Mélou \cite{bernardi} in the wider context of counting $q$-colorings of maps with respect to monochromatic edges, which is equivalent to computing the Potts partition function (the Potts model is a generalization of the Ising model to more than two colors).  
It is actually the algebraic function $Q_3(2,\nu,t)$ in \cite[Theorem 23]{bernardi}. 
Here, the parameter 2 refers to the number of colors, $t$ marks the number of edges and $\nu$ marks the number of monochromatic edges. 
Extracting the coefficient $[\nu^n]Q_3(2,\nu,t)$ we obtain a generating function which is equivalent to the generating function $M(z)$ of rooted cubic maps with a distinguished perfect matching, where $z$ marks the number of faces minus 2. 
After a simple algebraic manipulation, we obtain the following (we recall that the variable $z$ marks the number of faces minus two).

\begin{lemma}
	The generating function $M = M(z)$ counting rooted cubic planar maps with a distinguished perfect matching satisfies the quadratic equation
	\begin{equation*}
		72\,{M}^{2}{z}^{2} + \left( 216\, {z}^{2} - 36\,z + 1 \right) M + 162\,{z}^{2} - 6\,z = 0,
	\end{equation*}
	where the variable $z$ marks the number of faces minus two.
\end{lemma}
The former quadratic equation has a non-negative solution given by
$$
	M(z) = \frac{-1+36z-216z^2 + (1-24z)^{3/2}}{144z^2}.
$$
Expanding the binomial series, after a simple algebraic manipulation one obtains the simple formula 
\begin{equation}\label{eq:Mn}
	[z^n] M(z) = \frac{3\cdot 6^n}{(n+2)(n+1)} \binom{2n}{n}.
\end{equation}
In Section \ref{sec:bijections} we provide a direct combinatorial proof of this unexpected closed formula.

\begin{corollary}
The number $M_n$ of cubic planar maps with $2n$ vertices and a distinguished perfect matching is asymptotically  
$$
	\frac{3}{\sqrt{\pi}} n^{-5/2} 24^n.
$$
In addition, the expected number of perfect matchings in cubic planar maps with $2n$ vertices is asymptotically 
$$
	\frac{\sqrt{6}}{2} \left(\frac{2\sqrt{3}}{3}\right)^n.
$$
\end{corollary}

\begin{proof}
We use the notation $n!! = (n-1)(n-3) \dots$ for double factorials. 
The first claim follows directly from Stirling's estimate. 
The second claim follows since the number of cubic maps on $2n$ vertices is equal to 
$$
	\frac{2^{2n+1}(3n)!!}{(n+2)!n!!} \sim \frac{\sqrt{6}}{\sqrt{\pi}} n^{-5/2} (12\sqrt{3})^n,
$$
a result first proved in \cite{mullin2}. 
\end{proof}

In this simple case we have been able to deduce asymptotic estimates from simple closed formulas.
Later on, when we do not have closed formulas, we will need the full power of Lemma \ref{le:transfer}.

%---------------------------------------------------
\section{From cubic maps to cubic planar graphs}\label{sec:graphs}
%--------------------------------------------------

We start with a simple but very useful observation. 
We say that a class $\A$ of rooted maps is closed under rerooting if whenever a map is in $\A$, so is any map obtained from it by forgetting the root edge and choosing a different one.

\begin{lemma}\label{le:reroot}
Let $\A$ be a class of cubic planar maps closed under rerooting with a distinguished perfect matching. 
Let $\A_1$ be the maps in $\A$ whose root edge belongs to the perfect matching, and $\A_0$ those whose root edge does not belong to the perfect matching. 
Let $A_i(z)$ be the associated generating functions. 
Then $A_0(z) = 2A_1(z)$.
\end{lemma}

\begin{proof}
A cubic map with $2n$ vertices has $3n$ edges. 
Of these, $n$ of them are in the matching and $2n$ are not. 
Since $\N$ is closed under rerooting, the number of rooted maps whose root edge belongs to the matching is exactly half of those where it does not, hence $[z^n]A_0(z) = 2[z^n]A_1(z)$. 
\end{proof}

The previous lemma applies in particular to the class of all cubic maps and to the classes of 2-connected and  3-connected cubic maps.

\subsection{From cubic maps to 3-connected cubic maps}

In this section, we use the decomposition of cubic planar graphs as in \cite{bklm2007} and \cite{cubic-revisited}, adapted to cubic maps and enriched with a distinguished perfect matching, to obtain implicitely the generating function $T(z)$ of 3-connected cubic maps with a distinguished perfect matching. 
The classes defined in Section \ref{sec:prelim} are extended to cubic maps with a distinguished perfect matching.  
Maps in a given class are divided into those whose root edge is or is not in the matching. 
We use indices 1 and 0, respectively, to denote them, as in Lemma \ref{le:reroot}. 
For instance, $\mathcal{D}_1$ are cubic maps whose root edge is not an isthmus and belongs to the perfect matching. 

\begin{lemma}\label{le:EquationsMaps}
	The following system of equations holds and has a unique solution in power series with non-negative coefficients.
	\begin{equation}\label{system-maps}
		\renewcommand{\arraystretch}{1.6}
		\begin{array}{llllll}
			M_0(z) &=& D_0(z), & \quad
			M_1(z) &=& D_1(z) + I(z), \\
			D_0(z) &=& L(z) + S_0(z) + P_0(z) + H_0(z), & \quad
			D_1(z) &=& S_1(z) + P_1(z) + H_1(z), \\
			I(z) &=& \ds\frac{L(z)^2}{4z}, &\quad
			L(z) &=& 2z(1 + D_0(z)), \\
			S_1(z) &=& D_1(z)(D_1(z) - S_1(z)), &\quad
			S_0(z) &=& D_0(z)(D_0(z) - S_0(z)), \\
			P_1(z) &=& z(1 + D_0(z))^2, &\quad
			P_0(z) &=& 2z(1 + D_0(z))(1 + D_1(z)),\\ 
			H_1(z) &=& \ds\frac{T_1(z(1 + D_1(z))(1 + D_0(z))^2)}{1 + D_1}, &\quad
			H_0(z) &=& \ds\frac{T_0(z(1 + D_1(z))(1 + D_0(z))^2)}{1 + D_0}.
		\end{array}
	\end{equation}
\end{lemma}

Before we go to the proof, we notice that an isthmus must be in every perfect matching. 
This applies in particular to loop maps, which contain an isthmus incident with the root edge.

\begin{proof}
We start with an observation. 
An edge $e$ is replaced with a map whose root edge is in a perfect matching if and only if the two new edges resulting from the subdivision and replacement of $e$ belong to the resulting perfect matching. 

We know sketch the justification of each equation, the arguments are very similar to those in \cite{bklm2007}.
The equations for $M_i(z)$ and $D_i(x)$ follow from the definitions. 
The one for $I(z)$ is because an isthmus map is composed of two loop maps, as illustrated in Figure \ref{fig:series}; division by 4 takes into account the possible rootings of the two loops. 
The situation for $L(z)$, $S_i(z)$, $P_i(z)$ and $H_1(z)$ are rather straight forward and are also illustrated in Figure \ref{fig:series}.
However, the equations for $H_i$ can detailed as follows: in a cubic map with $2n$ vertices there are $n$ edges in a perfect matching and $2n$ not in it, hence the term $(1 + D_1(z))(1 + D_0(z))^2$ in the substitution.  
 
To show that there exists a unique solution with non-negative coefficients, we use Lemma \ref{positive}. 
It is enough to rewrite \eqref{system-maps} as a system with non-negative coefficients. 
To this end, replace $D_1 - S_1$ by $P_1 + H_1$ and $D_0 - S_0$ by $L + P_0 + H_0$. 
Notice also that $T_i(z)$ is divisible by $z^2$, so the equations for $H_i$ are polynomials with non-negative coefficients. 
\end{proof}

\begin{figure}[htb]
	\centering
	\includegraphics[scale=1]{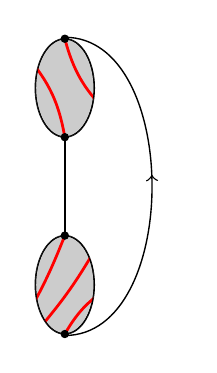}
	\hfill
	\includegraphics[scale=1]{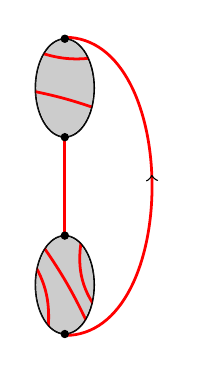}
	\hfill
	\includegraphics[scale=1]{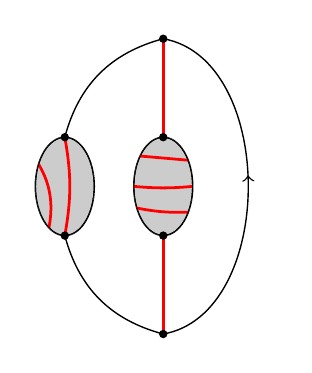}
	\hfill
	\includegraphics[scale=1]{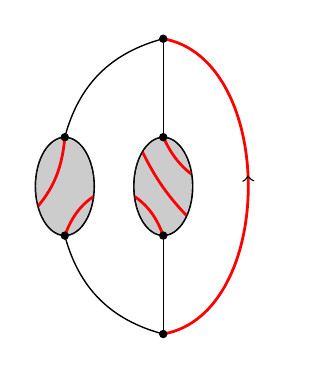} \\
	\raisebox{2.1cm}{\includegraphics[scale=1, angle=-90]{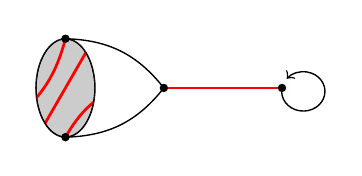}}
	\hfill
	\raisebox{.35cm}{\includegraphics[scale=1]{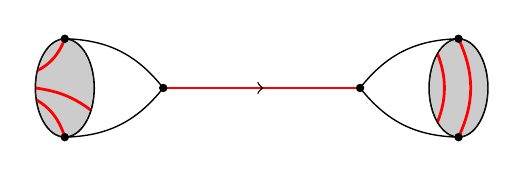}}
	\hfill
	\includegraphics[scale=1]{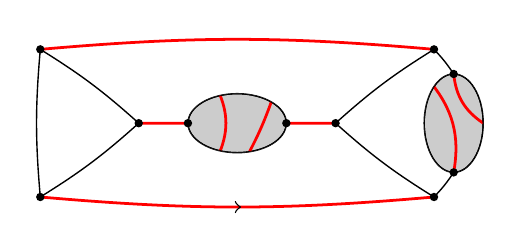}
\caption{
	Typical decompositions of series, parallel, loop, isthmus and polyhedral maps with a distinguished perfect matching shown in red.
	The associated generating functions are from left to right, top to bottom: $S_0(z)$, $S_1(z)$, $P_0(z)$, $P_1(z)$, $L(z)$, $I(z)$ and $H_1(z)$.
}\label{fig:series}
\end{figure}

\noindent
By algebraic elimination\footnote{These and subsequent computations have been performed with \texttt{Maple.}}, we obtain the minimal polynomial equation satisfied by $T_1=T_1(z)$, and hence the ones satisfied by $T_0(z)=2T_1(z)$ and $T(z)=T_0(z)+T_1(z)$:
\begin{equation}\label{eq:T}
	\begin{array}{ll}
		\,T_1^6 + \left( 24\,z+16 \right) T_1^5 + \left( 60\,{z}^{2}+92\,z+25 \right) T_1^4 + \left( 80\,{z}^{3}+208\,{z}^{2}+96\,z+19 \right) T_1^3\\
		+ \left( 60\,{z}^{4} + 232\,{z}^{3} + 150\,{z}^{2} + 12\,z + 7 \right) T_1^2 + \left( 24\,{z}^{5}+128\,{z}^{4} + 112\,{z}^{3} 
		+ {z}^{2} - 16\,z + 1 \right)T_1\\
		+ \,4\,{z}^{6} + 28\,{z}^{5} + 33\,{z}^{4} + 12\,{z}^{3} - {z}^{2} = 0.
	\end{array}
\end{equation}
The first terms of $T(z)$ are 
$$
	T(z) = 3x^2 + 12x^3 +69x^4 + 468x^5 + \cdots
$$
For instance, the first monomial corresponds to $K_4$, which has a unique rooting and 3 perfect matchings, and the second one to the triangular prism which has 3 different rootings and 4 perfect matchings. 

\subsection{Bridgeless cubic maps}\label{sec:bridgelessMaps}

It is also possible to obtain the generating function $B(z)$ of bridgeless cubic maps with a distinguished perfect matching. 
It suffices to remove isthmus and loop maps, which are those producing cut vertices. 
We rewrite the system \eqref{system-maps} without the series $I(z)$ and $L(z)$, and where $D_i(z)$, $S_i(z)$ and $P_i(z)$ now correspond to bridgeless maps:
$$
	\renewcommand{\arraystretch}{1.6}
	\begin{array}{llllll}
		D_0(z) &=&  S_0(z) + P_0(z) + H_0(z), & \quad
		D_1(z) &=& S_1(z) + P_1(z) + H_1(z), \\
		S_0(z) &=& D_0(z)(D_0(z) - S_0(z)), & \quad
		S_1(z) &=& D_1(z)(D_1(z) - S_1(z)), \\
		P_0(z) &=& 2z(1 + D_0(z))(1 + D_1(z)), & \quad
		P_1(z) &=& z(1 + D_0(z))^2,\\ 
		H_0(z) &=& \ds\frac{T_0(z(1 + D_1(z))(1 + D_0(z))^2)}{1 + D_0}, & \quad
		H_1(z) &=& \ds\frac{T_1(z(1 + D_1(z))(1 + D_0(z))^2)}{1 + D_1}.
	\end{array}
$$
Since we have already obtained the minimal polynomials of $T_0(z)$ and $T_1(z)$, we can eliminate and obtain the equation satisfied by $B(z) = D_0(z) +D_1(z)$, which is 
\begin{equation}\label{eq:B}
	\begin{array}{ll}
		64\,{B(z)}^{4}{z}^{3} + \left( 384\,{z}^{3} + 144\,{z}^{2} \right) {B(z)}^{3} 
		+ \left( 864\,{z}^{3} + 1224\,{z}^{2} + 108\,z \right) {B(z)}^{2} \\
		+ \left( 864\,{z}^{3} + 2700\,{z}^{2} - 756\,z + 27 \right) B(z) + 324\,{z}^{3} + 1782\,{z}^{2} - 81\,z = 0.
	\end{array}
\end{equation}
Using the \texttt{gfun} package \cite{gfun} in \texttt{Maple},
% \footnote{
% 	See \texttt{https://www.maplesoft.com/support/help/Maple/view.aspx?path=gfun} \\
% 	or \texttt{http://perso.ens-lyon.fr/bruno.salvy/software/the-gfun-package/}
% }
from  \eqref{eq:B} we have been able to find a first order recurrence equation with polynomial coefficients satisfied by the coefficients of $B(z)$ and, after some algebraic manipulations, we obtain the following.

\begin{lemma}
	The number $B_n$ of rooted cubic planar maps with $2n$ vertices and a distinguished perfect matching is given by
	\begin{equation}\label{eq:B_n}
		B_n = [z^n] B(z) = \frac{3\cdot2^{n-1}}{(2n+1)(n+1)} \binom{4n+2}{n}.
	\end{equation}
\end{lemma}

As for Equation \eqref{eq:Mn}, we give a combinatorial proof of this unexpected simple formula in Section \ref{sec:bijections}.

\begin{corollary}
	The number $B_n$ of bridgeless cubic maps with $2n$ vertices and a distinguished perfect matching is asymptotically  
	$$
		\frac{4\sqrt{6}}{9\sqrt{\pi}} n^{-5/2} \left( \frac{512}{27} \right)^n.
	$$
	In addition, the expected number of perfect matchings in bridgeless cubic maps is asymptotically 
	$$
		\frac{16\sqrt{2}}{9} \left(\frac{1024}{729}\right)^n.
	$$
\end{corollary}

\begin{proof}
The first claim follows again from Stirling's estimate, and the second since the number of bridgeless cubic maps is equal to 
$$
\frac{2^{n+1}}{(2n+2)(2n+1)} \binom{3n}{n} \sim \frac{\sqrt{3}}{4\sqrt{\pi}} n^{-5/2} \left(\frac{27}{2}\right)^n,
$$
a result first proved in \cite{tutte2}.
\end{proof}
We conclude this section with a table of numerical values for cubic maps up to 20 vertices. We remark that for 3-connected cubic maps there seems to be no simple closed formula.

\begin{table}[h]
\centering
\small
	\begin{tabular}{c|rrr}
	\toprule
		$n$ & $M_n$ & $B_n$ & $T_n$ \\
	\midrule
		2 & 6 & 3 & 0 \\
		4 & 54 & 18 & 3 \\
		6 & 648 & 156 & 12 \\
		8 & 9072 & 1632 & 69 \\
		10 & 139968 & 19152 & 468 \\
		12 & 2309472 & 242880 & 3582 \\
		14 & 40030848 & 3257280 & 29592 \\
		16 & 720555264 & 45568512 & 258561 \\
		18 & 13363024896 & 658910208 & 2356644 \\
		20 & 253897473024 & 9784140288 & 22201410 \\
		% 11 & 4921704861696 & 148488717312 & 214786344 \\
		% 12 & 97027895844864 & 2295013711872 & 2123788914 \\
		% 13 & 1940557916897280 & 36024518369280 & 21386558664 \\
		% 14 & 39296297817169920 & 573041175920640 & 218723224380 \\
		% 15 & 804418331786772480 & 9221227346903040 & 2266869752112 \\
		% 16 & 16624645523593297920 & 149895312435707904 & 23767091031321 \\
		% 18 & 7276344851273782394880 & 40645882586337902592 & 2690045619620634 \\
		% 19 & 153842719712645684920320 & 676797892977654497280 & 28977149882207112 \\
		% 20 & 3272654219341735479214080 & 11342103965779183534080 & 314377844809977654 \\
		% 21 & 70006342431136254598840320 & 191186496656262105661440 & 3432746039397711960 \\
  %       22 & 1505136362269429473875066880 & 3239829134389729705328640 & 37701751206679091940 \\
  %       23 & 32510945425019676635701444608 & 55168358077143425545666560 & 416278214977161456720 \\
  %       24 & 705237431527349908559062106112 & 943598834012141807072182272 & 4618594202151054160890 \\
  %       25 & 15358504064373398008619574755328 & 16205450420564799779757883392 & 51471416338894086418728 \\
	\bottomrule
	\end{tabular}
\caption{The numbers of (rooted) cubic planar maps $M_n$, bridgeless cubic planar maps $B_n$, and 3-connected cubic planar maps $T_n$, on
	% $n+2$ faces
$n$ vertices with a distinguished perfect matching.}
\label{tab:maps}
\end{table}

\subsection{From 3-connected cubic maps to cubic planar graphs}

A cubic \emph{network} is a connected cubic planar multigraph $G$ with an ordered pair of adjacent vertices $(s,t)$  such that the graph obtained by removing one of the edges between $s$ and $t$ is connected and simple.
We notice that $st$ can be a simple edge, a loop or a be part of a double edge, but cannot be an isthmus.
The oriented edge $st$ is called the \emph{root} of the network, and $s,t$ are called the \emph{poles}. 
Replacement in networks is defined as for maps. 

We let $\mathcal{D}$ be the class of cubic networks. 
The classes $\mathcal{I}$, $\mathcal{L}$, $\mathcal{S}$, $\mathcal{P}$ and $\mathcal{H}$ have the same meaning as for maps, and so do the subindices 0 and 1. 
We let $\mathcal{C}$ be the class of connected cubic planar graphs (always with a distinguished perfect matching), with associated generating function $C(x)$, and $C^\bullet(x) = xC'(x)$ be the generating functions of those graphs rooted at a vertex. 
We also let $G(x)$ be the generating function of arbitrary (non-necessarily connected) cubic planar graphs. 

Whitney's theorem claims that a 3-connected planar graph has exactly two embeddings in the sphere up to homeomorphism.
Thus counting 3-connected planar graphs rooted at a directed edge amounts to counting 3-connected maps, up to a factor 2. 
Below we use the notation $T_i(x)$ for the exponential generating functions of 3-connected cubic planar graphs rooted at a directed edge, similarly to maps.
% The difference with respect to the $T_i$ between the system \eqref{system-graphs} below and \eqref{system-maps} above is a factor of 2 in the denominator of the last two equations. 
% This is due to the fact that given an edge-rooted  3-connected planar graph there are two ways of producing a 3-connected map by choosing the face to the right of the root edge. 

The next result connects the unknown series $C^\bullet(x)$ with the series $T_0(z)$ and $T_1(z)$ obtained in the previous section. 

\begin{lemma}\label{le:system-graphs}
	The following system of equations holds and has a unique solution in power series with non-negative coefficients.
	\begin{equation}\label{system-graphs}
		\renewcommand{\arraystretch}{1.6}
		\begin{array}{llllll}
			D_0(x) &=& L(x) + S_0 (x)+ P_0(x) + H_0(x), & \quad
			D_1(x) &=& S_1(x) + P_1(x) + H_1(x), \\
			L(x) &=& \ds\frac{x^2}2(D_0(x) - L(x)), & \quad
			I(x) &=& \ds\frac{L(x)^2}{x^2}, \\
			S_0(x) &=& D_0(x)(D_0(x) - S_0(x), & \quad
			S_1(x) &=& D_1(x)(D_1(x) - S_1(x)), \\
			P_0(x) &=& x^2(D_0(x) + D_1(x)) + x^2D_0(x)D_1(x), & \quad
			P_1(x) &=& x^2D_0(x) + \ds\frac{x^2}2D_0(x)^2, \\
			H_0 &=& \ds\frac{T_0(x^2(1 + D_1)(1 + D_0^2))}{2(1 + D_0(x))}, & \quad
			H_1(x) &=& \ds\frac{T_1(x^2(1 + D_1)(1 + D_0^2))}{2(1 + D_1(x))}.
		\end{array}
\end{equation}
Moreover, we have
$$
	3C(x)^{\bullet} = I(x) + D_0(x) + D_1(x) - L(x) - L^2(x) - 2x^2D_0(x) - x^2D_1(x).
$$
\end{lemma}

\begin{proof}
The first part of the proof is very similar to that of Lemma \ref{le:EquationsMaps} and is omitted. 
The last equation for $C^\bullet(x)$ follows by double counting, since rooting at a directed edge in a cubic planar graph is equivalent to rooting at a vertex and selecting one of its 3 incident edges.
The positive terms $I(x) + D_0(x) + D_1(x)$ correspond to all possible networks. 
Since we are counting simple graphs we subtract those producing loops or multiple edges, i.e. graphs rooted at a loop and those where the root edge is a double edge: the parallel networks encoded by $2x^2D_0(x) + x^2D_1(x)$, and the series networks encoded by $L(x)^2$.

To prove the uniqueness of the solution, we use Lemma \ref{positive} as in the proof of Lemma \ref{le:EquationsMaps}.
\end{proof}

%----------------------------------------
\section{Proofs of the main results}\label{sec:results}
%----------------------------------------

\paragraph{Proof of Theorem \ref{th:all}.}
We apply the transfer theorem as described in Section \ref{sec:prelim}.
We need to find the dominant singularity of $C(x)$, which is the same as that of $D_0(x)$, $D_1(x)$ and then $D(x)$.
It is obtained by first computing the minimal polynomial for $D(x)$ and then its discriminant $\Delta(x)$. 
After discarding several factors of $\Delta(x)$ for combinatorial reasons (as in \cite{cubic-revisited}), the relevant factor of $\Delta(x)$ turns out to be  
$$
	904x^8 + 7232x^6 - 11833x^4 - 45362x^2 + 3616,
$$
whose smallest positive root is equal to $\sigma \approx 0.27964$. 
After routinely checking the conditions of \cite[Lemma 15]{cubic-revisited}, we conclude that $\sigma$ is the only positive dominant singularity and that $D(x)$ admits an expansion near $\sigma$ of the form 
$$
	D(x) = d_0 + d_2X^2 + d_3X^3 + O(X^4), \qquad X = \sqrt{1 - x/\sigma}.
$$
And the same hold for $D_0(x)$ and $D_1(x)$. 
But also for $L(x)$ and $I(x)$, using their definitions given in terms of $D_0(x)$ in Lemma \ref{le:system-graphs}.

There is a second singularity $-\sigma$ with a similar singular expansion and, as explained in \cite{cubic-revisited}, the contributions of $\pm\sigma$ are added using a straightforward extension of Lemma \ref{le:transfer}. 
This is also the case in the next proof and we omit the details to avoid repetition.
 
From there, and using again Lemma \ref{le:system-graphs} we can compute the singular expansion of $C^\bullet(x) = xC'(x)$, and by integration, that of $C(x)$. 
For arbitrary cubic planar graphs, we use the exponential formula  $G(x) = e^{C(x)}$, which encodes the fact a graph is an unordered set of connected graphs. 
The transfer theorem finally gives 
\begin{equation}\label{eq:Gn}
	G_n = [x^n] G(x) \approx c_1 n^{-7/2} \sigma^{-n} n!. 
\end{equation}
To obtain the expected value of $X_n$ we have to divide $G_n$ by the number $g_n$ of labeled cubic planar graphs, which as shown in \cite{bklm2007,cubic-revisited} is asymptotically $g_n \sim c_0 n^{-7/2} \rho^{-n}n!$, where $c_0 >0$ and $\rho \approx 0.31923$ is the smallest positive root of 
$$
729x^{12}+17496x^{10}+148716x^8+513216x^6-7293760x^4+279936x^2+46656=0.
$$
And we obtain the claimed result by setting $c= c_1/c_0$ and $\delta = \rho/\sigma$. 
Furthermore, since $\sigma$ and $\rho$ are algebraic numbers, so is $\delta$ (actually of degree 48).

For the second part of the statement we argue as follows. 
Since a graph with $n$ vertices has at most $n!$ automorphisms, the number of unlabeled graphs in a class is at least the number of labeled graphs divided by $n!$. 
It follows that the number $U_n$ of unlabeled cubic planar graphs with a distinguished perfect matching is at least $G_n/n!$, where $G_n$ is given in \eqref{eq:Gn}. 

No precise estimate is known for the number $u_n$ of unlabeled cubic planar graphs, but it can be upper bounded by the number $C_n$ of simple rooted cubic planar \emph{maps}, because a planar graph has at least one embedding in the plane. 
These maps have already benn counted in \cite{cubicMaps} and the estimate $C_n \sim c_s \cdot n^{-5/2} \alpha^{-n}$, where $\alpha \sim 0.3102$,  follows from \cite[Corollary 3.2]{cubicMaps}. 
The relation between $\alpha$ and the value $x_0$ given in \cite{cubicMaps} is $\alpha = x_0^{1/2}$; this is due to the fact that we count cubic maps according to faces whereas in \cite{cubicMaps} they are counted according to vertices, and a map with $n+2$ faces has $2n$ vertices.
Disregarding subexponential terms, we have $U_n \ge \sigma^{-n}$ and $u_n \le \alpha^{-n}$. 
The last result holds as claimed since $\alpha/\sigma \approx 1.109$. \qed

\paragraph{Proof of Theorem \ref{th:bridgeless}.}
The proof follows the same scheme as that of Theorem \ref{th:all}.
Similarly to Section \ref{sec:bridgelessMaps} we have to adapt the system \eqref{system-graphs} to bridgeless cubic planar graphs. 
To this end, we remove the generating functions $I(z)$ and $L(z)$ and obtain the system:
\begin{equation}\label{system-bridgeless-graphs}
	\renewcommand{\arraystretch}{1.6}
	\begin{array}{llllll}
		D_0(x) &=& S_0 (x)+ P_0(x) + H_0(x), &
		D_1(x) &=& S_1(x) + P_1(x) + H_1(x), \\
		S_0(x) &=& D_0(x)(D_0(x) - S_0(x)), &
		S_1(x) &=& D_1(x)(D_1(x) - S_1(x)), \\
		P_0(x) &=& x^2(D_0(x) + D_1(x)) + x^2D_0(x)D_1(x), &
		P_1(x) &=& x^2D_0(x) + \ds\frac{x^2}2D_0(x)^2, \\
		H_0(x) &=& \ds\frac{T_0(x^2(1 + D_1)(1 + D_0^2))}{2(1 + D_0(x))}, &
		H_1(x) &=& \ds\frac{T_1(x^2(1 + D_1)(1 + D_0^2))}{2(1 + D_1(x))},
	\end{array}
\end{equation}

\bigskip
\noindent
where $D_i(x)$, $S_i(x)$, $P_i(x)$ and $H_i(x)$ now refer to bridgeless cubic networks. 
If $A^\bullet(x)=xA'(x)$ is the generating function of bridgeless cubic planar graphs rooted at a vertex  (we avoid using the more natural letter $B$ to avoid confusion with  bridgeless \emph{maps}),
an argument analogous to that in the proof of Lemma \ref{le:system-graphs} gives the relation 
$$
	3A(x)^{\bullet} = D_0(x) +D_1(x) - 2x^2D_0(x)  - x^2D_1(x).
$$
The relevant factor of the discriminant of $D(x)$ is now equal to 
$$
	216x^6 + 864x^4 - 5587x^2 + 432,
$$
whose smallest positive root is the dominant singularity $\sigma_b \approx 0.27980$. 
Again, we have a singular expansion of the form 
$$
	D(x) = d_0 + d_2X^2 + d_3X^3 + O(X^4), \qquad \text{with } X = \sqrt{1-x/\sigma_b},
$$
and the same holds for $D_0(x)$ and $D_1(x)$.

If we let $A_n = [x^n] A(x)$ then the transfer theorem gives 	$A_n  \sim r_1 n^{-7/2} \sigma_b^{-n} n!$.
To obtain the expected value of $Y_n$ we have to divide $A_n$  by the number $b_n$ of labeled cubic planar graphs. It is computed asymptotically in \cite{bklm2007,cubic-revisited}, and is given by $b_n \sim r_0 n^{-7/2} \rho_b^{-n}n!$, where $r_0 >0$ and $\rho_b \approx 0.319523$ is the smallest positive root of 
$$
	54x^6+324x^4-4265x^2+432=0.
$$
We obtain the claimed result setting $b = r_1/r_0$ and $\gamma = \rho_b/\sigma_b$. 
And since $\sigma_b$ and $\rho_b$ are algebraic numbers, so is $\gamma$ (of degree 18).

The bound for unlabeled graphs is derived from the estimate $r_s n^{-5/2} \alpha_b^{-n}$ on the number of simple bridgeless cubic maps, with $\alpha_b = ((3\sqrt{3}-5)/2)^{1/2} \approx 0.31317$, as computed in \cite{cubicMaps}. 
The result finally follows since $\alpha_b/\sigma_b \approx 1.119$. \qed

Again we conclude  with a table of numerical values for cubic planar graphs with up to 20 vertices. 
\begin{table}[h]
	\centering
	\small
	\begin{tabular}{c|rrr}
		\toprule
		$n$ & $G_n$ & $C_n$ & $A_n$ \\
		\midrule
		4 & 3 & 3 & 3 \\
		6 & 240 & 240 & 240 \\
		8 & 70875 & 70560 & 70560 \\
		10 & 39795840 & 39644640 & 39191040 \\
		12 & 36909890325 & 36778341600 & 36119714400 \\
		14 & 51164374781520 & 50994240897600 & 49863647433600 \\
		16 & 99734407245898425 & 99424934088480000 & 96928583719968000 \\
		18 & 260680626187437456000 & 259925179413803904000 & 252809307842547456000 \\
		20 & 881248549547808635868675 & 878853675324753063936000 & 853158542751301602816000 \\
		%22 & 3745018149248113875775638000 & 3735427559504549457669120000 & 3620183325822727054110720000 \\
		% 24 & 19553207451731576075187502042275 & 19505781078285140498543523840000 & 18875547625201942294493153280000 \\
		% 26 & 123086044564123719819076691715360000 & 122801898964289761978568142336000000 & 118667576288201600458667860992000000 \\
		% 28 & 919586317272446030524948048815375103125 & 917556666841347098497072545051648000000 & 885483743406597556662208855127040000000 \\
		% 30 & 8046004119935703953828226830363715561810000 & 8028955327909329147551549328818325504000000 & 7738341281041539514765262020403466240000000 \\
		\bottomrule
	\end{tabular}
	\caption{The numbers of (labeled) cubic planar graphs $G_n$, connected cubic planar graphs $C_n$, and bridgeless cubic planar graphs $A_n$, on $n$ vertices with a distinguished perfect matching.}
	\label{tab:graps}
\end{table}

%----------------------------------------------------
\section{Bijective proofs}\label{sec:bijections}
%---------------------------------------------------
In this section we provide bijective proofs for  the expressions obtained earlier by algebraic computations  in Equations \eqref{eq:Mn} and \eqref{eq:B_n}.

\paragraph{Cubic maps with a perfect matching.}

Let $S_n$ be the number of rooted cubic maps with $2n$ vertices and a  distinguished perfect matching such that the root edge is in the matching. Since exactly $1/3$ of the cubic maps are such that their root edge belongs to the matching (see Lemma \ref{le:reroot}),  Equation \eqref{eq:Mn} gives
$$
	S_n = 6^n \frac{\binom{2n}{n}}{(n+2)(n+1)}.
$$

We recall that the number $R_n$ of rooted planar maps with $n$ edges (see \cite{tutte3}) is equal to 
$$
	R_n = 2\cdot 3^n \frac{\binom{2n}{n}}{(n+2)(n+1)}.
$$
It is well-known that the number $R_n$ also counts rooted 4-regular maps with $n$ vertices \cite{tutte3}.
From Equation \eqref{eq:Mn} it follows that $S_n =  2^{n-1} R_n$. 
In order to prove this relation combinatorially, we show the following.

\begin{lemma}\label{le:bijectionM}
	There exists a  $2^{n-1}:1$ correspondence between rooted cubic planar maps on $2n$ vertices, with a distinguished perfect matching containing  the root edge, and rooted 4-regular planar maps on $n$ vertices.
\end{lemma}

\begin{proof}
	Let $M$ be a cubic map with a distinguished perfect matching $A$ in which the root edge $r$ is in $A$. 
	Contract (respecting the embedding) the edges in $A$ to obtain a 4-regular map $F$. 
	We root $F$ at the edge next to $r$ in the outer face.
	
	The reverse mapping is as follows. 
	Take a 4-regular map $F$ with $n$ vertices.
	Splitting a vertex $v$ means to replace it by two vertices $u$ and $u'$, connected via a new edge, and make a binary choice as how to distribute the four pendant half-edges between $u$ and $u'$, while preserving the embedding; see Figure \ref{fig:split}.
	Now split each non-root vertex in $F$ in two possible ways to obtain a cubic map $M$, where the split edges are those in the matching. 
	The root vertex is split in only one way so that the root edge of $M$ is in the matching. 
	Since there are $n-1$ non-root vertices, this concludes the proof.
\end{proof}

\begin{figure}[htb]
			\centering
	\begin{minipage}{.35\textwidth}
		\includegraphics[scale=1]{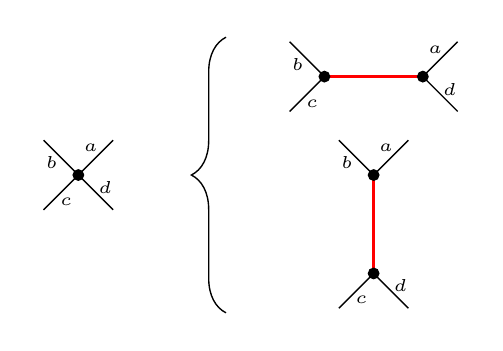}
	\end{minipage}
	\hfill\vline\hfill
	\begin{minipage}{.64\textwidth}
				\centering
		\includegraphics[scale=1]{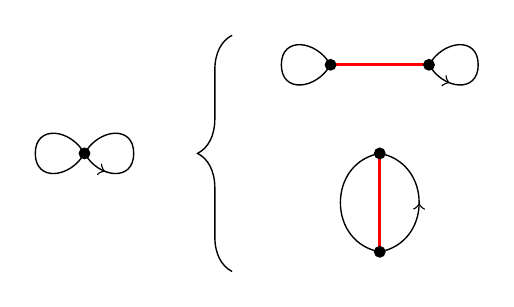}
		\qquad
		\includegraphics[scale=1]{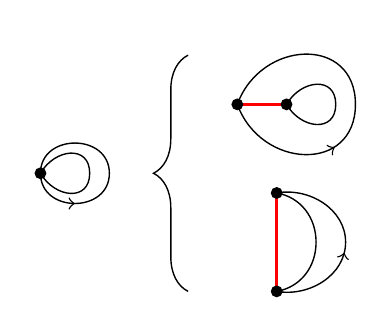}
	\end{minipage}
	\caption{
		On the left, the two ways of splitting a vertex of degree 4 respecting the embedding.
		On the right we display the four cubic maps on two vertices with a distinguished perfect matching not containing the root edge, obtained from the two 4-regular maps on one vertex by splitting.   
		The edges in each of the matchings are in red.
	}
	\label{fig:split}
\end{figure}

\paragraph{Bridgeless cubic maps with a perfect matchings.}

Let $L_n =\frac{1}{(2n+1)(n+1)}\binom{4n+2}{n}$.
As before, from Equation \eqref{eq:B_n} and Lemma \ref{le:reroot} it follows that the number of rooted bridgeless cubic maps with $2n$ vertices and a  distinguished perfect matching containing the root edge is equal to 
$$
	\frac{1}{3} B_n = 2^{n-1} L_n.
$$
The sequence $L_n$  is well-known and  equal to the number of all rooted loopless planar maps on $n$ edges \cite{walsh}, and by duality is equal to the number of rooted bridgeless planar maps on $n$ edges (interestingly it is also equal to the number of 3-connected triangulations with $n$ internal vertices). 
Thus to prove the former relation, it suffices to show the following.
The bijection in this case is more involved that the one in the proof of Lemma \ref{le:bijectionM}.

\begin{lemma}\label{le:bijectionN}
	There exists a  $2^{n-1}:1$ correspondence between bridgeless rooted cubic planar maps on $2n$ vertices with a distinguished perfect matching containing the root edge, and rooted (arbitrary) bridgeless planar maps on $n$ edges.
\end{lemma}

\begin{proof}
Given a rooted bridgeless map $B$ with $n$ edges, we construct a cubic map $M$ with $2n$ vertices as follows. 
Consider the edges of $B$ as consisting of two half-edges.
Replace each vertex $v$ of degree $k$ in $B$ by a $k$-cycle with $k$ half-edges attached to it, pair the half-edges outside the cycles while respecting the embedding in $B$, and root $M$ at the same directed edge as $B$. 
This construction (illustrated in Figure \ref{fig:flip}) goes back to Tutte, and is the pendant for planar maps of the \emph{truncation} of the vertices of a polytope in discrete geometry.
It is clear that $M$ is a cubic map and that the original edges of $B$ form a perfect matching in $M$, containing the root edge.  
In what follows, the edges in the matching are called \emph{red} edges, while the remaining edges are \emph{black}.
It is also clear that if $M$ contains a bridge, so does $B$, hence $M$ is bridgeless. 
Also, notice that a loop in $B$ gives rise to a double edge in $M$ containing exactly one red edge.

\begin{figure}[htb]
	\centering
	\begin{minipage}{.44\textwidth}
		\underline{\textit{Case 1}} \\
		\includegraphics[scale=1]{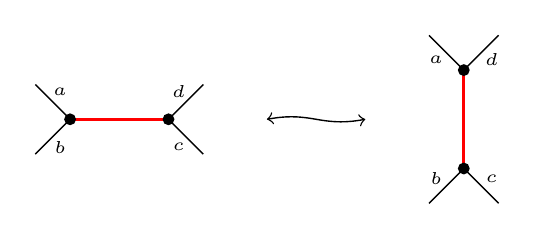}\\
		~\\~\\
		\underline{\textit{Case 2}} ($\{b,c\}$ disconnects the map)\\
		\includegraphics[scale=1]{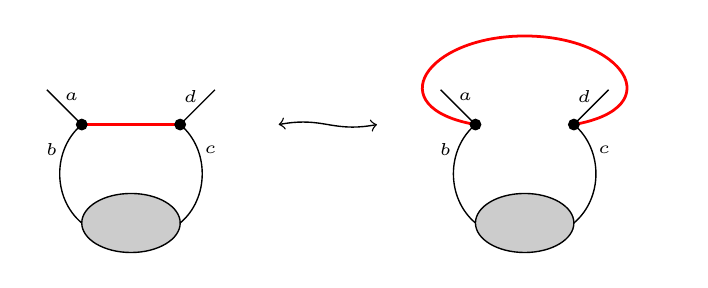}	
	\end{minipage}
	\hfill\vline\hfill
	\begin{minipage}{.55\textwidth}
		\centering
		\begin{tabular}{ccc}
			\underline{\textit{Bridgeless}} & \underline{\textit{Truncation}} & \underline{\textit{Flip}} \\
			\raisebox{.2cm}{\includegraphics[scale=1]{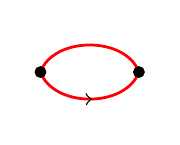}} &
			\includegraphics[scale=1]{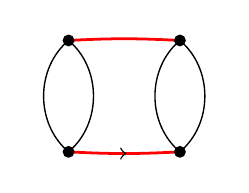} &
			\includegraphics[scale=1]{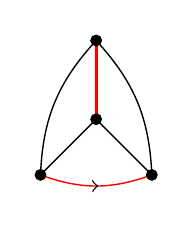} \\
 			\raisebox{.2cm}{\includegraphics[scale=1.2]{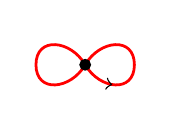}} &
			\includegraphics[scale=1]{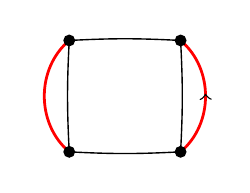} &
			\includegraphics[scale=1]{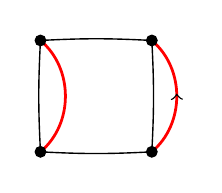} \\
			\raisebox{1cm}{\includegraphics[scale=1.2]{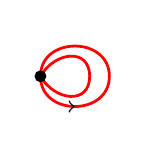}} &
			\includegraphics[scale=1]{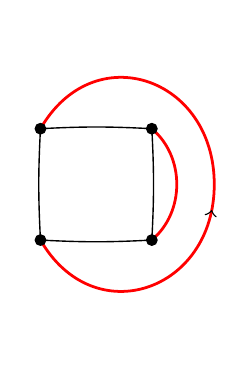} &
			\includegraphics[scale=1]{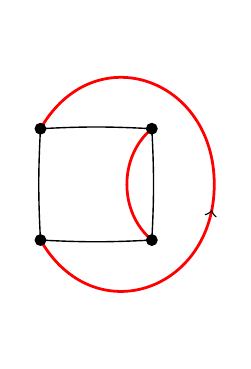}
		\end{tabular}
	\end{minipage}
	\caption{
		On the left we illustrate the flip rule on red edges. 
		On the right we show the six bridgeless cubic maps on four vertices, each with a perfect matching (in red) containing the root, produced from the three bridgeless maps on two edges by expansion then the possible flip of each non-root red edge.
	}
	\label{fig:flip}
\end{figure}	

In order to obtain $2^{n-1}$ different cubic maps with a perfect matching, from the original bridgeless map $B$, we define a \emph{flip operation} on the red edges. 
It is performed on every non-root red edge of $M$ as follows.
If $xy$ is a red edge and $(a,b)$ are the black half-edges incident with $x$ in counterclockwise order, and $c,d$ those with $y$, then we pair $b$ with $c$ and $a$ with $d$, as in Figure \ref{fig:flip}, and keep the red edge $xy$. 
There is an exception to this rule, though: if the pair of half-edges $\{b,c\}$ either forms an edge or disconnectsd the map, then we remove $xy$ and redraw it inside the corners $(a,b)$ and $(c,d)$, so that the cyclic ordering becomes $(b,a,d,c)$; see Figure \ref{fig:flip}. 
Notice that in the latter case, $\{a,d\}$, just like $\{b,c\}$, forms either an edge or disconnects the map.
Since there are $n-1$ red edges besides the root edge, and the flip operation is an involution, we obtain exactly $2^{n-1}$ different bridgeless cubic maps with a perfect matching containing the root edge.
We remark that several red edges can be flipped simultaneously and the order in which they are flipped is irrelevant. 

To conclude the proof, it remains to show that in this process every bridgeless cubic map with a perfect matching containing the root edge is produced exactly once. 
To this end, we show next how to recover the original map $B$ from a cubic map $M$ endowed with a perfect matching. 
If we remove the red edges from $M$, we are left with a 2-regular graph, that is, a collection of cycles, that we call the \emph{black cycles}. 
Notice that since $M$ is cubic, every red edge is incident with either one or two black cycles. 
Every black cycle delimits two regions, the \emph{outside} region is the one containing the root edge and the other one is the \emph{inside} region. 
A red edge is said to be \emph{bad} if it lies inside a black cycle, otherwise it is \emph{good}.
A bad red edge is called \emph{worthy} if it lies inside a unique black cycle $C$ and has at least one endpoint in $C$. 
Suppose $M$ has a bad edge $e$. 
Then by considering the outermost black cycle containing $e$, there must exist a worthy bad edge in $M$. 

We say that a cubic map $M$ is \emph{good} if it contains no bad edge.  
We want to show that we can flip red edges, in any cubic map, until we obtain a good map. 
If $M$ is not good then, as argued before, it has at least one bad worthy edge $e=xy$. 
Let $(a,b,c,d)$ be the half edges incident with $xy$ as in Figure \ref{fig:flip}. 
We now consider the two cases in the definition of the flip operation. 
Let $M'$ be the map obtained from $M$ by flipping $e$. 
\begin{itemize}
	\item If neither $\{b,c\}$ nor $\{a,d\}$ form an edge or disconnect $M$, then we have two cases: either both $x$ and $y$ are in the same cycle black $C$, or they belong to different black cycles $C$ and $C'$.
	In the first case, $C$ is replaced by two cycles $C_1$ and $C_2$ in $M'$ which are incident, respectively, with $x$ and $y$, and $e$ lies outside both $C_1$ and $C_2$. 
	Since $e$ was contained only in $C$, it ceases to be bad. 
	In the second case, since $e$ is worthy, one of the two cycles, say $C$, is in the interior of the other one $C'$. 
	Then $C$ and $C'$ are replaced with a cycle $C''$, incident with both $x$ and $y$, and with $e$ outside $C''$. 
	In both cases, $e$ has become a good edge of $M'$, and since the new cycles created are all contained in $C$, no good edge of $M$ has become a bad edge in $M'$.
	
	\item If $\{a,d\}$ disconnects $M$, so does $\{b,c\}$, so we can assume the latter. 
	In this case, $x$ and $y$ are necessarily both in $C$. 
	After flipping $e$, the cycle $C$ does not change and now $e$ is outside it. 
	Hence $e$ has become a good edge in $M'$, and again no new bad edge has been created.
\end{itemize}
\noindent
Thus, in both cases the number of bad edges has decreased when passing from $M$ to $M'$, as we wanted to prove. 
It remains to show that the good map obtained after this process is unique. 
Suppose that from a map $M$ we can reach two different good maps $M_1$ and $M_2$, by flipping red edges. 
Then $M_1$ can be transformed into $M_2$ by a sequence of edge flips.
But flipping a good edge transforms it into a bad worthy edge, as it is easily checked. 
This is a contradiction since $M_2$ contains only good edges. 

In conclusion, from $M$ we can reach a unique good map $M'$. 
Now contracting all the black cycles in $M'$ (none of them containing a red edge), we recover the original bridgeless map $B$.
\end{proof}

%----------------------------------------------------
\section{Concluding remarks}
%----------------------------------------------------

We have obtained an alternative derivation of Equation \eqref{eq:T} that we find interesting in itself. 
It is based on the enumeration of 3-connected 4-regular maps, a recent result obtained in \cite{4-regular,4-regular-asympt}. 
In particular, an equation is found in \cite{4-regular-asympt} for the generating functions $S_i(u,v)$ of 3-connected 4-regular maps, where $u$ marks the number of simple edges and $v$ marks the number of double edges, and the index is $i=1$ if the root edge is simple, and $i=2$ if the root edge is double (they are named $T_i$ in \cite{4-regular-asympt}, we have renamed $S_1$ and $S_2$ to distinguish them from the $T_0$ and $T_1$ introduced in this paper). 
Double edges in a 3-connected map are vertex disjoint, since if two double edges share a vertex $v$, then the other two endpoints would disconnect $v$. 
It follows that in a 3-connected 4-regular maps with $2n$ vertices, the maximum number of double edges is $n$. 
In this case, every vertex is incident with exactly one double edge, and there are $2n$ simple edges. 
Now from every double edge remove one edge and mark the other edge to obtain a cubic map with a distinguished perfect matching. 
The process is reversible, hence the generating function $T_0(z)$ from this paper can be recovered as the diagonal of $S_1(u,v)$ and setting $z=uv$.

A natural open question is to prove some kind of concentration result for the number of perfect matchings. 
But already computing the variance seems out of reach with our techniques, since for computing the second moment we would need to consider maps or graphs with a pair of distinguished perfect matchings, and the connection with the Ising model does not seem to provide this. 
In fact, we believe that to compute the variance for the number of global structures in non-trivial classes of graphs defined by a global condition, such as being planar, is in general a challenging question. 
Let us mention that the variance has been computed for the number of perfect matchings in regular graphs \cite{bollobas}, using the so-called configuration model, showing that there is no concentration. 
The situation is very different for the binomial model $G(n,p)$ of random graphs with $n$ vertices and where edges are selected independently with probability $p$: in this circumstance, not only concentration but a central limit theorem has been proved for the number of perfect matchings and other spanning subgraphs \cite{janson}. 

We also mention that the number of perfect matchings in a cubic graph with $2n$ vertices is at most $6^{n/3} \approx 1.817^n$ \cite[Theorem 8.24]{lovasz}; see also \cite{alon}. 
We are not aware of a sharper upper bound for cubic planar graphs. 
On the other hand, the number of perfect matchings in the prism $C_n \times  K_2$ grows like the sum of two Fibonacci numbers, so there are cubic planar graphs with $2n$ vertices and $\phi^n$ perfect matchings, where $\phi=(1+\sqrt{5})/2 \approx 1.618$.

\paragraph{Acknowledgements.}
The third author acknowledges preliminary discussions on this topic with Mihyun Kang, Michael Mosshammer and Philipp Sprüssel during a visit  to the Technical University of Graz in 2015, and wishes to thank M.K. for her invitation.

%----------------------------------------------------
\bibliographystyle{abbrv}
\bibliography{biblio_PM_cubic}

\begin{thebibliography}{10}

\bibitem{alon}
N.~Alon and S.~Friedland.
\newblock The maximum number of perfect matchings in graphs with a given degree
  sequence.
\newblock {\em Electronic Journal of Combinatorics}, 15(N13):2, 2008.

\bibitem{bernardi}
O.~Bernardi and M.~Bousquet-M\'elou.
\newblock Counting coloured planar maps: {A}lgebraicity results.
\newblock {\em Journal of Combinatorial Theory, Series B}, 101(5):315--377,
  2011.

\bibitem{bklm2007}
M.~Bodirsky, M.~Kang, M.~L\"offler, and C.~McDiarmid.
\newblock Random cubic planar graphs.
\newblock {\em Random Structures \& Algorithms}, 30(1-2):78--94, 2007.

\bibitem{bollobasRG}
B.~Bollob\'as.
\newblock {\em Random {G}raphs. {S}econd edition}.
\newblock Cambridge Studies in Advanced Mathematics. Cambridge University
  Press, 2011.

\bibitem{bollobas}
B.~Bollob\'as and B.~D. McKay.
\newblock The number of matchings in random regular graphs and bipartite
  graphs.
\newblock {\em Journal of Combinatorial Theory, Series B}, 41(1):80--91, 1986.

\bibitem{eulerian}
M.~Bousquet-M\'elou and A.~Elvey~Price.
\newblock The generating function of planar {E}ulerian orientations.
\newblock {\em Journal of Combinatorial Theory, Series A}, 172:48, 2020.

\bibitem{brezin}
E.~Br\'ezin, C.~Itzykson, G.~Parisi, and J.-B. Zuber.
\newblock Planar diagrams.
\newblock {\em Communications in Mathematical Physics volume}, 59:35--51, 1978.

\bibitem{chudnovsky}
M.~Chudnovsky and P.~Seymour.
\newblock Perfect matchings in planar cubic graphs.
\newblock {\em Combinatorica}, 32:403--424, 2012.

\bibitem{DVW}
A.~Denise, M.~Vasconcellos, and D.~J.~A. Welsh.
\newblock The random planar graph.
\newblock {\em Congressus Numerantium}, 113:61--79, 1996.

\bibitem{diestel}
R.~Diestel.
\newblock {\em Graph {T}heory. {F}ifth edition}, volume 173 of {\em Graduate
  Texts in Mathematics}.
\newblock Springer-Verlag Berlin Heidelberg, 2017.

\bibitem{DrmotaBook}
M.~Drmota.
\newblock {\em Random trees}.
\newblock SpringerWienNewYork, Vienna, 2009.
\newblock An interplay between combinatorics and probability.

\bibitem{universal}
M.~Drmota, M.~Noy, and G.-R. Yu.
\newblock Universal singular exponents in catalytic variable equations.
\newblock Submitted. arXiv:2003.07103.

\bibitem{esperet}
L.~Esperet, F.~Kardo\u{s}, A.~D. King, D.~Kr\'al, and S.~Norine.
\newblock Exponentially many perfect matchings in cubic graphs.
\newblock {\em Advances in Mathematics}, 227(4):1646--1664, 2011.

\bibitem{esperet2}
L.~Esperet, F.~Kardo\u{s}, and D.~Kr\'al.
\newblock A superlinear bound on the number of perfect matchings in cubic
  bridgeless graphs.
\newblock {\em European Journal of Combinatorics}, 33(5):767--798, 2012.

\bibitem{eynard}
B.~Eynard.
\newblock {\em Counting {S}urfaces. {CRM} {A}isenstadt chair lectures},
  volume~70 of {\em Progress in Mathematical Physics}.
\newblock Birkh\"auser Basel, 2016.

\bibitem{baxter}
S.~Felsner, E.~Fusy, and M.~Noy.
\newblock Asymptotic enumeration of orientations.
\newblock {\em Discrete Mathematics \& Theoretical Computer Science},
  12(2):249--262, 2010.

\bibitem{fs}
P.~Flajolet and R.~Sedgewick.
\newblock {\em Analytic {C}ombinatorics}.
\newblock Monograph. Cambridge University Press, 2009.

\bibitem{cubicMaps}
Z.~Gao and N.~C. Wormald.
\newblock Enumeration of rooted cubic planar maps.
\newblock {\em Annals of Combinatorics}, 6:313--325, 2002.

\bibitem{gn}
O.~Gim\'{e}nez and M.~Noy.
\newblock Asymptotic enumeration and limit laws of planar graphs.
\newblock {\em Journal of the American Mathematical Society}, 22(2):309--329,
  2009.

\bibitem{janson}
S.~Janson.
\newblock The numbers of spanning trees, {H}amilton cycles and perfect
  matchings in a random graph.
\newblock {\em Combinatorics, Probability and Computing}, 3(1):97--126, 1994.

\bibitem{lando}
S.~K. Lando and A.~K. Zvonkin.
\newblock {\em Graphs on {S}urfaces and {T}heir {A}pplications. {W}ith an
  appendix by {D}on {B}. {Z}agier}, volume 141 of {\em Encyclopaedia of
  Mathematical Sciences. Low Dimensional Topology, II}.
\newblock Springer-Verlag Berlin Heidelberg, 2004.

\bibitem{lovasz}
L.~Lov\'asz and M.~D. Plummer.
\newblock {\em Matching Theory}, volume~29 of {\em Annals of discrete
  mathematics}.
\newblock North-Holland: Elsevier Science, 1986.

\bibitem{mullin}
R.~C. Mullin.
\newblock On the enumeration of tree-rooted maps.
\newblock {\em Canadian Journal of Mathematics}, 19:174--183, 1967.

\bibitem{mullin2}
R.~C. Mullin, E.~Nemeth, and P.~J. Schellenberg.
\newblock The enumeration of almost cubic maps.
\newblock In {\em Proceedings of the Louisiana Conference on Combinatorics,
  Graph Theory and Computing, 1970. Louisiana State University, Baton Rouge,
  Louisiana}, pages 281--295, 1970.

\bibitem{noy}
M.~Noy.
\newblock Random planar graphs and beyond.
\newblock In K.~M. Sa, editor, {\em Proceedings of the International Congress
  of Mathematicians - Seoul, 2014}, volume~IV, pages 407--430, 2014.

\bibitem{4-regular-asympt}
M.~Noy, C.~Requil\'e, and J.~Ru\'e.
\newblock Asymptotic enumeration of labelled 4-regular planar graphs.
\newblock Submitted. arXiv:2001.05943.

\bibitem{4-regular}
M.~Noy, C.~Requil\'e, and J.~Ru\'e.
\newblock Enumeration of labelled 4-regular planar graphs.
\newblock {\em Proceedings of the London Mathematical Society},
  119(2):358--378, 2019.

\bibitem{cubic-revisited}
M.~Noy, C.~Requil\'e, and J.~Ru\'e.
\newblock Further results on random cubic planar graphs.
\newblock {\em Random Structures \& Algorithms}, 56(3):892--924, 2020.

\bibitem{petersen}
J.~Petersen.
\newblock Die theorie der regul\"aren graphs.
\newblock {\em Acta Mathematica}, 15:193--220, 1891.

\bibitem{gfun}
B.~Salvy and P.~Zimmermann.
\newblock Gfun: a {M}aple package for the manipulation of generating and
  holonomic functions in one variable.
\newblock {\em ACM Transactions on Mathematical Software}, 20(2):163--177,
  1994.

\bibitem{tait}
P.~G. Tait.
\newblock Remarks on the colourings of maps.
\newblock {\em Proceedings of the Royal Society of Edinburgh}, 10(4):501--503,
  1880.

\bibitem{tutteHam}
W.~T. Tutte.
\newblock On {H}amiltonian circuits.
\newblock {\em Journal of the London Mathematical Society}, 21(2):98--101,
  1946.

\bibitem{tutte2}
W.~T. Tutte.
\newblock A census of {H}amiltonian polygons.
\newblock {\em Canadian Journal of Mathematics}, 14:402--417, 1962.

\bibitem{tutte1}
W.~T. Tutte.
\newblock A census of planar triangulations.
\newblock {\em Canadian Journal of Mathematics}, 14:21--38, 1962.

\bibitem{tutte3}
W.~T. Tutte.
\newblock A census of planar maps.
\newblock {\em Canadian Journal of Mathematics}, 15:249--271, 1963.

\bibitem{tutteBook}
W.~T. Tutte.
\newblock {\em Graph theory as I have known it}, volume~11 of {\em Oxford
  Lecture Series in Mathematics and Its Applications}.
\newblock Oxford University Press, Oxford, 2012.

\bibitem{valiant}
L.~G. Valiant.
\newblock The complexity of computing the permanent.
\newblock {\em Theoretical Computer Science}, 8(2):189--201, 1979.

\bibitem{walsh}
T.~R.~S. Walsh and A.~B. Lehman.
\newblock Counting rooted maps by genus {III}: {N}onseparable maps.
\newblock {\em Journal of Combinatorial Theory, Series B}, 18(3):222--259,
  1975.

\end{thebibliography}

\end{document}